\documentclass[a4paper,11pt]{amsart}

\usepackage{mathpazo}
\usepackage{amssymb}
\usepackage[pagebackref,
    ,pdfborder={0 0 0}
    ,urlcolor=black,a4paper,hypertexnames=false]{hyperref}
\hypersetup{pdfauthor=Clara L\"oh,pdftitle=Finite functorial semi-norms and representability}

\usepackage{pgf}
\usepackage{tikz}
\usetikzlibrary{decorations.pathmorphing}
\usepackage[all]{xy}
\SelectTips{cm}{}

\newtheorem{thm}{Theorem}[section]
\newtheorem{prop}[thm]{Proposition}
\newtheorem{lem}[thm]{Lemma}
\newtheorem{cor}[thm]{Corollary}

\newtheorem{setup}[thm]{Setup}

\theoremstyle{remark}
\newtheorem{rem}[thm]{Remark}
\newtheorem{exa}[thm]{Example}

\theoremstyle{definition}
\newtheorem{defi}[thm]{Definition}

\usepackage{amsfonts}
\newcommand{\Z}{\mathbb{Z}}
\newcommand{\Q}{\mathbb{Q}}
\newcommand{\R}{\mathbb{R}}
\newcommand{\N}{\mathbb{N}}

\DeclareMathOperator{\id}{id}
\DeclareMathOperator{\Hom}{Hom}

\DeclareMathOperator{\map}{map}
\DeclareMathOperator{\Chf}{Ch}
\def\epsilon{\varepsilon}

\DeclareMathOperator{\Cst}{{\sf C}\,\text{\raisebox{-.1em}{${}^*$}}}
\DeclareMathOperator{\Set}{{\sf Set}}
\DeclareMathOperator{\Top}{{\sf Top}}

\DeclareMathOperator{\Toppth}{{\sf Top}_{* \sf h}}

\DeclareMathOperator{\Group}{{\sf Group}}
\DeclareMathOperator{\Ab}{{\sf Ab}}

\DeclareMathOperator{\CWh}{{\sf CW}_{{\sf h}}}

\DeclareMathOperator{\snAb}{{\sf snAb}}
\DeclareMathOperator{\snAbfin}{{\sf snAb}^{{\sf fin}}}
\DeclareMathOperator{\Ch}{{\sf Ch}}
\DeclareMathOperator{\CoCh}{{\sf CoCh}}

\def\RCh#1{\Ch_{#1}}
\DeclareMathOperator{\Ob}{Ob}
\DeclareMathOperator{\Mor}{Mor}

\newcommand{\loneh}{H^{\ell^1}}

\def\args{\;\cdot\;}

\def\exi#1{%
  \exists_{#1}\;\;\;%
}

\usepackage{color}
\usepackage{pdfcolmk}

\def\draftinfo{}

\author{Clara L\"oh}
\title{Finite functorial semi-norms and representability}
\date{\today.\ \copyright{\ C.~L\"oh 2014}. 
    This work was supported by the CRC~1085 \emph{Higher Invariants} 
    (Universit\"at Regensburg).
    \draftinfo\\
     MSC~2010 classification: 55N10, 55N35, 57N65}
\begin{document}

\begin{abstract}
  Functorial semi-norms are semi-normed refinements of functors such
  as singular (co)homology. We investigate how different types of
  representability affect the (non-)triviality of finite functorial
  semi-norms on certain functors or classes. In particular, we
  consider representable functors, generalised cohomology theories,
  and so-called weakly flexible homology classes in singular homology
  and $\ell^1$-homology.
\end{abstract}

\phantom{.}
\vspace{-.1\baselineskip}

\maketitle

\section{Introduction}

  Functorial semi-norms are semi-normed refinements of functors to the
  category of Abelian groups (Definition~\ref{defffsn}). Gromov
  introduced the notion of functorial semi-norms on singular
  (co)homology~\cite{gromov, vbc} in the context of simplicial volume
  and its relation with geometry and rigidity. For example, the
  $\ell^1$-semi-norm on singular homology is a finite functorial
  semi-norm on singular homology (Example~\ref{exal1top}).  Dually,
  the $\ell^\infty$-semi-norm on singular cohomology
  (Example~\ref{exalinftytop}) is a functorial semi-norm on singular
  cohomology, which is not finite; in contrast, by construction, the
  $\ell^\infty$-semi-norm on bounded cohomology (Example~\ref{exabc})
  is a finite functorial semi-norm on bounded cohomology.
  Applications of functorial semi-norms in manifold topology include
  degree theorems \cite{vbc,lafontschmidt,loehsauer}. Conversely,
  knowledge about mapping degrees allows to construct functorial
  semi-norms with interesting properties~\cite{crowleyloeh,
    kotschickloeh}.

  We investigate the relation between (non-)triviality of finite
  functorial semi-norms and so-called weak flexibility:

  \begin{defi}[weakly flexible]
    Let $C$ be a category, let $F \colon C \longrightarrow \Ab$ be a
    functor, and let $X \in \Ob(C)$. An element~$\alpha \in F(X)$ is \emph{weakly
      flexible (with respect to~$F$)} if there exist~$Y \in
    \Ob(C)$ and~$\beta \in F(Y)$ such that the set
    \[ D(\beta,\alpha) 
       := \bigl\{ d \in \Z 
       \bigm| \exi{f \in \Mor_C(Y,X)}
       F(f)(\beta) = d \cdot \alpha 
       \bigr\}  
    \]
    is infinite. 
  \end{defi}
  
  It is a simple, yet fundamental, observation that any finite
  homogeneous  functorial semi-norm is trivial on weakly flexible classes
  (Proposition~\ref{propfund}).
  
  In this article, we give equivalent descriptions of weak flexibility
  in singular homology (Section~\ref{subsecwfsing}) and
  $\ell^1$-homology (Section~\ref{subsecwfl1}). Moreover, we apply the
  above observation to representable functors
  (Section~\ref{subsecwfrep}) and countably additive functors 
  (Section~\ref{subsecwfgc}).  This shows that many classical functors
  from algebraic topology do \emph{not} admit any non-trivial finite
  homogeneous functorial semi-norms. For instance, this might be interesting for the study of
  comparison maps in bounded cohomology.

\section{Functorial semi-norms}\label{secbasics}

In this section, we introduce some basic notation and give a
definition of functorial semi-norms, generalising Gromov's concept of
functorial semi-norms on singular (co)homology:

\begin{defi}[semi-norms on Abelian groups]
  \hfil
  \begin{itemize}
    \item A \emph{semi-norm} on an Abelian group~$A$  
      is a map~$|\cdot| \colon A \longrightarrow \R_{\geq 0} \cup \{\infty\}$ 
        with the following properties: 
        \begin{itemize}
          \item  We have~$|0| = 0$.
          \item For all~$x \in A$ we have~$|-x| = |x|$.
          \item For all~$x , y \in A$ we have~$|x + y| \leq |x| + |y|$. Here, 
            $x + \infty := \infty$ for all~$x \in \R_{\geq 0} \cup \{\infty\}$.
        \end{itemize}
        Such a semi-norm is \emph{finite} if the value~$\infty$ is not
        in the image. The semi-norm is \emph{homogeneous} if for
        all~$n \in \Z \setminus \{0\}$ and all~$x \in A$ we have~$|n
        \cdot x| = |n| \cdot |x|$, where~$|n|$ denotes the ordinary
        absolute value on~$\Z$.  Here, $n \cdot \infty := \infty$ for
        all~$n \in \N \setminus \{0\}$.
    \item A \emph{(finitely) semi-normed Abelian group} is an Abelian group 
      together with a (finite) semi-norm.
    \item We write $\snAb$ for the category of semi-normed Abelian
      groups, where the morphisms are group homomorphisms that are
      norm-non-increasing.
    \item We write $\snAbfin$ for the category of finitely semi-normed
      Abelian groups, where the morphisms are group homomorphisms that
      are norm-non-increasing.
  \end{itemize}
\end{defi}

\begin{defi}[functorial semi-norm]\label{defffsn}
  Let $C$ be a category, let $F \colon C \longrightarrow \Ab$ 
  be a functor (possibly contravariant).
  \begin{itemize}
    \item A \emph{functorial semi-norm on~$F$} is a factorisation
      functor~$\widehat F \colon C \longrightarrow \snAb$ of~$F$ 
      through the forgetful functor~$\snAb \longrightarrow \Ab$:
      \begin{align*}
        \xymatrix{%
          & 
          \snAb \ar[d]^{\txt{forget}}
          \\
          C \ar[r]_{F} \ar@{-->}[ur]^{\widehat F}
          & \Ab
        }
      \end{align*}
    \item A \emph{finite functorial semi-norm on~$F$} is a factorisation
      $\widehat F \colon C \longrightarrow \snAbfin$ of~$F$ 
      through the forgetful functor~$\snAbfin \longrightarrow \Ab$:
      \begin{align*}
        \xymatrix{%
          & 
          \snAbfin \ar[d]^{\txt{forget}}
          \\
          C \ar[r]_{F} \ar@{-->}[ur]^{\widehat F}
          & \Ab
        }
      \end{align*}
  \end{itemize}
\end{defi}

More explicitly, a [finite] functorial semi-norm on a functor~$F
\colon C \longrightarrow \Ab$ consists of a choice of a [finite]
semi-norm~$|\cdot|$ on~$F(X)$ for every~$X \in \Ob(C)$ such that for
all morphisms~$f \colon X \longrightarrow Y$ in~$C$ and all~$\alpha
\in F(X)$ we have
\[ \bigl|F(f)(\alpha)\bigr| \leq |\alpha|.
\]
A functorial semi-norm on~$F$ is \emph{homogeneous} if for all~$X \in
\Ob(C)$ the corresponding semi-norm on~$F(X)$ is homogeneous.
Notice that all homogeneous (functorial) semi-norms are trivial on torsion
elements.

\begin{exa}[trivial functorial semi-norm]
  If $C$ is a category and $F \colon C \longrightarrow \Ab$ is a
  functor, then equipping for every~$X \in \Ob(C)$ the group~$F(X)$
  with the zero semi-norm gives a functorial semi-norm on~$F$, the 
  \emph{trivial functorial semi-norm on~$F$}.
\end{exa}

\subsection{$\ell^1$-Homology}

A key example is given by the $\ell^1$-semi-norm on homology of
simplicial sets; the $\ell^1$-semi-norm on singular homology and the
$\ell^1$-semi-norm on group homology are instances of this
construction:
 
\begin{exa}[$\ell^1$-semi-norm on ($\ell^1$)-homology of simplicial sets]
  \label{exal1simpset}
  Let $\Delta$ be the \emph{simplex category}: The objects of~$\Delta$
  are natural numbers; for~$n, m \in \N$ the set of morphisms~$n
  \longrightarrow m$ in~$\Delta$ is the set of monotone (increasing)
  functions of type~$\{0,\dots,n\} \longrightarrow \{0,\dots,m\}$, and
  composition of morphisms in~$\Delta$ is given by the set-theoretic
  composition of the underlying functions. The category of
  \emph{simplicial sets} is denoted by~$\Delta(\Set)$; objects are
  contravariant functors~$\Delta \longrightarrow \Set$ and morphisms
  are natural transformations of such functors. 
  If $R$ is a ring (with unit), then there is the corresponding chain
  complex functor
  \[ \Chf_R \colon \Delta(\Set) \longrightarrow \RCh R 
                              \longrightarrow \RCh \Z
  \]
  into the category of $R$-chain complexes~\cite{goerssjardine} followed by
  the forgetful functor into $\Z$-chain comlexes. By construction, for
  a simplicial set~$S$ and~$n\in \N$, the chain module~$\Chf_R(S)_n$
  is the free $R$-module with basis~$S(n)$.  If $R$ (as Abelian group)
  is equipped with a norm, then $\Chf_R(S)_n$ inherits a finite
  norm~$|\cdot|_{1,R}$: the $\ell^1$-norm with respect to the
  basis~$S(n)$. This norm in turn leads to the
  \emph{$\ell^1$-semi-norm on homology} in degree~$n$:
  \begin{align*} 
    \|\cdot\|_{1,R} \colon H_n\bigl(\Chf_R(S)\bigr) & \longrightarrow \R_{\geq 0}\\
    \alpha & \longmapsto \inf \bigl\{ |c|_{1,R} 
             \bigm| c \in \Chf_R(S)_n,\ \partial c= 0,\  
             [c] = \alpha \bigr\}.
  \end{align*}

  By definition of~$|\cdot|_{1,R}$, the functor~$\Chf_R$ turns natural
  transformations into norm-non-increasing chain maps. Applying
  homology in degree~$n$ hence yields a finite functorial semi-norm on the
  composition
  \[ H_n \circ \Chf_R \colon \Delta(\Set) \longrightarrow \Ab.
  \]
  
  By construction of the chain complex functor~$\Chf_R$, the boundary
  operators in the chain complex~$\Chf_R(S)$ are bounded with respect to
  the $\ell^1$-norm. Hence, completing with respect to~$|\cdot|_{1,R}$
  (and extending the boundary operators continuously) leads to a
  functor
  \[ \Chf_R^{\ell^1} \colon \Delta(\Set) \longrightarrow \RCh \Z.
  \]
  Clearly, the $\ell^1$-norm on the chain modules of~$\Chf_R(S)$ extends
  to an~$\ell^1$-norm on the chain modules of~$\Chf_R^{\ell^1}(S)$,
  which we will also denote by~$|\cdot|_{1,R}$. Similarly to the
  uncompleted case, we then obtain a finite functorial $\ell^1$-semi-norm on
  the \emph{$\ell^1$-homology} functor
  \[ H_n^{\ell^1} \circ \Chf_R 
     = H_n \circ \Chf_R^{\ell^1} \colon \Delta(\Set)  \longrightarrow
     \Ab
  \]
  in degree~$n$. 

  In all these cases, we abbreviate~$\|\cdot\|_1 := \|\cdot\|_{1,\R}$.
  Notice that if $\Q$ is a subring of~$R$, then these functorial semi-norms 
  on ($\ell^1$-)homology are homogeneous. In general, $\|\cdot\|_{1.\Z}$ 
  is \emph{not} homogeneous.
\end{exa}

\begin{exa}[$\ell^1$-semi-norm on singular homology]\label{exal1top}
  Let $R$ be a normed ring, and let $n \in \N$. Then singular
  homology~$H_n(\args;R)$ in degree~$n$ of topological spaces with
  coefficients in~$R$ can be described as the composition
  \[ H_n \circ \Chf_R \circ S \colon 
    \Top \longrightarrow \Delta(\Set) \longrightarrow \RCh \Z \longrightarrow \Ab,
  \]
  where $S \colon \Top \longrightarrow \Delta(\Set)$ denotes the total
  singular complex functor~\cite{goerssjardine}. So, $H_n(\args;R)$ inherits a finite 
  functorial semi-norm, the $\ell^1$-semi-norm~\mbox{$\|\cdot\|_{1,R}$}.  More explicitly, $\|\cdot\|_{1,R}$ 
  is given as follows: If $X$ is a topological space and~$\alpha \in H_n(X;R)$, 
  then
  \[ \|\alpha\|_{1,R} = \inf \biggl\{ \sum_{j=0}^k |a_j|
     \biggm| \text{$\sum_{j = 0}^k a_j \cdot \sigma_j \in C_n(X;R)$ is a cycle representing~$\alpha$}
     \biggr\},
  \]
  where $|\cdot|$ denotes the norm on~$R$. 

  An example of a topological invariant defined in terms of the
  $\ell^1$-semi-norm on singular homology with $\R$-coefficients is
  the simplicial volume~\cite{vbc, mapsimvol}: The \emph{simplicial
    volume} of an oriented closed connected $n$-manifold~$M$ is
  \[ \|M\| := \bigl\| [M]_\R \bigr\|_1 \in \R_{\geq 0},
  \]
  where $[M]_\R \in H_n(M;\R)$ is the $\R$-fundamental class of~$M$.
  For example, the simplicial volume of hyperbolic manifolds is
  non-zero~\cite{vbc,thurston,benedettipetronio}.  Classical applications of simplicial volume
  include Gromov's proof of Mostow rigidity~\cite{munkholm} and degree
  theorems in the presence of sufficiently negative
  curvature~\cite{vbc,lafontschmidt,loehsauer}. 
\end{exa}

Functorial semi-norms on singular homology that differ essentially
from the $\ell^1$-semi-norm can be constructed via manifold
topology~\cite{crowleyloeh, kotschickloeh}. However, it remains an
open problem to determine whether there are also \emph{finite}
functorial semi-norms on singular homology that are not dominated in
some weak sense (e.g., through a monotonic function) by the
$\ell^1$-semi-norm.

\begin{exa}[$\ell^1$-semi-norm on group homology]\label{exal1groups}
  Let $R$ be a normed ring, and let $n \in \N$.  Similarly to the
  previous example, homology of groups in degree~$n$ with coefficients
  in~$R$ can be described as the composition
  \[ H_n \circ \Chf_R \circ B \colon 
     \Group \longrightarrow \Delta(\Set) \longrightarrow \RCh \Z \longrightarrow \Ab,
  \]
  where $B \colon \Group \longrightarrow \Delta(\Set)$ denotes the
  simplicial classifying space functor~\cite{goerssjardine}. Hence, also
  group homology in degree~$n$ with coefficients in~$R$ inherits a
  functorial $\ell^1$-semi-norm~$\|\cdot\|_{1,R}$. 

  The canonical isomorphism between group homology via
  the simplicial classifying space functor and singular homology of 
  Eilenberg-MacLane spaces of type~$K(\args,1)$ is isometric with 
  respect to~$\|\cdot\|_{1,R}$, as can be seen by constructing explicit 
  mutually inverse chain homotopy equivalences. 
\end{exa}

\begin{exa}[$\ell^1$-semi-norm on $\ell^1$-homology]\label{exal1l1}
  Let $R$ be a normed ring, and let $n \in \N$. Then
  \emph{$\ell^1$-homology of spaces or groups} respectively in
  degree~$n$ with coefficients in~$R$ are defined by
  \begin{align*}
    \loneh_n(\args;R) & := H^{\ell^1}_n \circ \Chf_R \circ S 
      \colon \Top \longrightarrow \Delta(\Set) \longrightarrow \Ab,\\
    \loneh_n(\args;R) & := H^{\ell^1}_n \circ \Chf_R \circ B 
      \colon \Group \longrightarrow \Delta(\Set) \longrightarrow \Ab.
  \end{align*}
  As in Example~\ref{exal1simpset}, these functors admit functorial
  $\ell^1$-semi-norms~$\|\cdot\|_{1,R}$. For more information about
  $\ell^1$-homology and its applications we refer to the
  literature~\cite{loehl1,buehler,bfs,loehpagliantini}.
\end{exa}

\subsection{Bounded cohomology}

Dually, we can also equip cohomology of simplicial sets with a
functorial semi-norm -- the $\ell^\infty$-semi-norm, which in turn 
leads to bounded cohomology:

\begin{exa}[$\ell^\infty$-semi-norm on (bounded) cohomology of simplicial sets]
  \label{exalinftysimpset}
  Let $R$ be a normed ring. Then there is the dual cochain 
  complex functor
  \[ \Chf_R^* := \Hom_\Z(\args,R) \circ \Chf_\Z
     \colon \Delta(\Set) \longrightarrow \CoCh_R \longrightarrow \CoCh_\Z
  \]
  as well as the topological dual cochain complex functor 
  \[ \Chf_R^\# := B(\args,R) \circ \Chf_\Z 
     \colon \Delta(\Set) \longrightarrow \CoCh_R \longrightarrow \CoCh_\Z
  \]
  of bounded $\Z$-linear functionals; both these functors are
  contravariant.  If $S$ is a simplicial set and $n \in \N$, then
  the modules~$\Chf_R^*(S)^n$ and $\Chf_R^\#(S)^n$ inherit a
  norm~$|\cdot|_{\infty,R}$: the $\ell^\infty$-norm on linear
  functionals with respect to the $\ell^1$-norm on~$\Chf_\Z(S)_n$;
  this norm on~$\Chf_R^*(S)^n$, in general, is not finite, but -- by
  construction -- this norm on~$\Chf_R^\#(S)^n$ is finite. These norms
  in turn lead to the \emph{$\ell^\infty$-semi-norm on 
    cohomology} in degree~$n$
  \begin{align*}
    \|\cdot\|_{\infty,R} 
    \colon H^n\bigl(\Chf_R^*(S)\bigr)
    & \longrightarrow \R_{\geq 0} \cup \{\infty\}
    \\
    \varphi &\longmapsto
    \inf \bigl\{ |f|_{\infty,R} \bigm| f \in \Chf_R^*(S)^n, 
                                     \delta f =0, [f] =\varphi\bigr\},
  \end{align*}
  and the \emph{$\ell^\infty$-semi-norm on bounded cohomology} in degree~$n$:
  \begin{align*}
    \|\cdot\|_{\infty,R} 
    \colon H^n\bigl(\Chf_R^\#(S)\bigr)
    & \longrightarrow \R_{\geq 0}
    \\
    \varphi &\longmapsto
    \inf \bigl\{ |f|_{\infty,R} \bigm| f \in \Chf_R^\#(S)^n, 
                                     \delta f =0, [f] =\varphi\bigr\}. 
  \end{align*}
  By definition of~$|\cdot|_{\infty,R}$, the functors~$\Chf_R^*$ and $\Chf_R^\#$ 
  turn natural transformations into norm-non-increasing chain maps. Applying 
  (co)homology in degree~$n$ hence yields a finite functorial semi-norm 
  on the contravariant compositions
  \begin{align*}
    H^n \circ \Chf_R^* \colon \Delta(\Set) & \longrightarrow \Ab, \\
    H^n \circ \Chf_R^\# \colon \Delta(\Set) & \longrightarrow \Ab.
  \end{align*}
  In all theses cases, we abbreviate~$\|\cdot\|_\infty :=
  \|\cdot\|_{\infty,\R}$.  If $\Q$ is a subring of~$R$, then these
  functorial semi-norms on (bounded) cohomology are homogeneous. In
  general, $\|\cdot\|_{\infty,\Z}$ is \emph{not} homogeneous.
\end{exa}

\begin{exa}[$\ell^\infty$-semi-norm on singular cohomology]\label{exalinftytop}
  Let $R$ be a normed ring, and let $n \in \N$. Then singular
  cohomology $H^n(\args;R)$ in degree~$n$ of topological spaces with
  coefficients in~$R$ can be described as the composition
  \[ H^n \circ \Chf_R^* \circ S \colon \Top \longrightarrow 
                               \Delta(\Set) \longrightarrow \CoCh_\Z
                               \longrightarrow \Ab,
  \]
  where $S \colon \Top \longrightarrow \Delta(\Set)$ denotes the total
  singular complex functor. So, $H^n(\args;R)$ inherits a functorial
  semi-norm, the $\ell^\infty$-semi-norm~$\|\cdot\|_{\infty,R}$. More
  explicitly, $\|\cdot\|_{\infty,R}$ is given as follows: If $X$ is a 
  topological space and $\varphi \in H^n(X;R)$, then
  \[ \|\varphi\|_{\infty,R} 
     = \inf \Bigl\{ \sup_{\sigma \in \map(\Delta^n,X)} |f(\sigma)| 
            \Bigm| 
                    \text{$f\in C^n(X;R)$ is a cocycle representing~$\varphi$}
            \Bigr\},
  \]
  where $|\cdot|$ denotes the norm on~$R$.
\end{exa}

\begin{exa}[$\ell^\infty$-semi-norm on group cohomology]\label{exlinftygroup}
  Let $R$ be a normed ring, and let $n \in \N$. As in 
  the previous examples, cohomology of groups in degree~$n$ with
  coefficients in~$R$ inherits a functorial
  $\ell^\infty$-semi-norm~$\|\cdot\|_{\infty,R}$, based on the
  simplicial classifying space functor~$B \colon \Group
  \longrightarrow \Delta(\Set)$.
\end{exa}

\begin{exa}[$\ell^\infty$-semi-norm on bounded cohomology]\label{exabc}
  Let $R$ be a normed ring, and let $n \in \N$. Then \emph{bounded cohomology 
    of spaces or groups} respectively in degree~$n$ with coefficients in~$R$ 
  are defined by
  \begin{align*}
    H^n_b(\args;R) & := 
    H^n \circ \Chf_R^\# \circ S \colon \Top \longrightarrow \Delta(\Set) 
    \longrightarrow \Ab 
    \\
    H^n_b(\args;R) & := 
    H^n \circ \Chf_R^\# \circ B \colon \Group \longrightarrow \Delta(\Set)
    \longrightarrow \Ab.
  \end{align*}
  As described in Example~\ref{exalinftysimpset}, these functors 
  admit \emph{finite} functorial $\ell^\infty$-semi-norms~$\|\cdot\|_{\infty,R}.$ 

  The inclusion of bounded linear functionals into all linear
  functionals induces natural transformations~$H^n_b(\args;R)
  \Longrightarrow H^n(\args;R)$ between bounded cohomology of
  spaces/groups and ordinary cohomology of spaces/groups, the
  so-called \emph{comparison maps}. In general, these comparison 
  maps are neither injective nor surjective (Example~\ref{exa:compmaps}).
\end{exa}

There is a duality principle that allows to express the
$\ell^1$-semi-norm on homology in terms of the $\ell^\infty$-semi-norm
of the corresponding bounded cohomology~\cite{vbc}. Therefore, bounded
cohomology is an important algebraic tool in the study of simplicial
volume. For more information about bounded
cohomology and its applications we refer to the
literature~\cite{vbc,ivanov,monod,buehler}.


\section{Weak flexibility}\label{secweakflex}

We introduce the following notions of representability of classes,
generalising (strong) inflexibility of
manifolds~\cite{crowleyloeh}:

\begin{defi}[(weakly) flexible]
  Let $C$ be a category, let $F \colon C \longrightarrow \Ab$ be a
  functor, let $X \in \Ob(C)$, and let $\alpha \in F(X)$.
  \begin{itemize}
    \item The element~$\alpha$ is \emph{flexible (with respect to~$F$)}
      if the set
      \[ D(\alpha) := \bigl\{ d \in \Z 
                        \bigm| \exists_{f \in \Mor_C(X,X)}\;\;\; 
                               F(f)(\alpha) = d \cdot \alpha 
                        \bigr\} 
      \]
      is infinite. Otherwise $\alpha$ is called \emph{inflexible}.
    \item The element~$\alpha$ is \emph{weakly flexible (with respect
      to~$F$)} if there exists an object~$Y \in \Ob(C)$ and~$\beta \in
      F(Y)$ such that the set 
      \[ D(\beta,\alpha) 
         := \bigl\{ d \in \Z 
                        \bigm| \exists_{f \in \Mor_C(Y,X)}\;\;\; 
                               F(f)(\beta) = d \cdot \alpha 
                        \bigr\}  
      \]
      is infinite. Otherwise $\alpha$ is called \emph{strongly 
        inflexible}.
  \end{itemize}
\end{defi}

\begin{exa}
For example, fundamental classes of spheres, tori, and (oriented)
projective spaces are flexible with respect to singular homology 
with $\Z$- or $\R$-coefficients. 
Not all (fundamental classes of) oriented closed simply connected manifolds are
flexible~\cite{arkowitzlupton,crowleyloeh,amann,costoyaviruel};
however, it remains an open problem whether there exist oriented
closed simply connected manifolds that are strongly inflexible.
\end{exa}

\begin{exa}[torsion classes]
  Let $C$ be a category and let $F \colon C \longrightarrow \Ab$ be a
  functor. If $X \in \Ob(C)$ and $\alpha \in F(X)$ is a torsion
  element of order~$m$, then $m \cdot \Z + 1\subset D(\alpha)$ and so
  $\alpha$ is flexible.
\end{exa}


In the following, we will make use of the following simple, yet
fundamental, observation:

\begin{prop}[weak flexibility and finite functorial semi-norms]
  \label{propfund}
  Let $C$ be a category, let $F \colon C \longrightarrow \Ab$ 
  be a functor, let $X \in \Ob(C)$ and let $\alpha \in F(X)$ be 
  weakly flexible. 
  If $\widehat F \colon C \longrightarrow \snAbfin$ 
  is a finite homogeneous functorial semi-norm on~$F$, then $|\alpha|_{\widehat F} = 0$.
\end{prop}
\begin{proof}
  Because $\alpha$ is weakly flexible, there is an object~$Y \in
  \Ob(C)$ and $\beta \in F(Y)$ such that $D(\beta,\alpha)$ is
  infinite. If $\widehat F \colon C \longrightarrow \snAbfin$ is a
  finite homogeneous functorial semi-norm on~$F$, we obtain
  \[ |\alpha|_{\widehat F} 
     \leq \inf \Bigl\{ \frac 1d \cdot |\beta|_{\widehat F} 
               \Bigm| d \in D(\beta, \alpha) \setminus \{0\}\Bigr\} = 0.
     \qedhere
  \]
\end{proof}

\begin{exa}
  In particular, the fundamental class of an oriented closed connected
  manifold with non-zero simplicial volume ist \emph{not} weakly
  flexible with respect to real singular homology.  Prominent examples
  of this type are hyperbolic
  manifolds~\cite{vbc,thurston,benedettipetronio}.
\end{exa}

\section{Representable functors and cohomology theories}

We will now apply the observation Proposition~\ref{propfund} to
representable functors (Section~\ref{subsecwfrep}) as well as 
countably additive functors (Section~\ref{subsecwfgc}).

\subsection{Weak flexibility and representable functors}\label{subsecwfrep}

Representable functors do not admit non-trivial finite functorial semi-norms:

\begin{cor}[finite functorial semi-norms on representable functors]
  \label{corffsnrep}
  Let $C$ be a category and let $F \colon C \longrightarrow \Ab$ be a
  (co- or contravariant) representable functor. If $\widehat F \colon
  C \longrightarrow \snAbfin$ is a finite homogeneous functorial semi-norm on~$C$,
  then $\widehat F$ is trivial.
\end{cor}
\begin{proof}
  We give the proof in the covariant case (the contravariant case
  being dual). By Proposition~\ref{propfund}, it suffices to establish 
  that all classes are weakly flexible with respect to~$F$. 
  Let $Y \in \Ob(C)$ be a representing object of~$F$,
  i.e., $V \circ F \cong \Mor_C(Y,\args)$, where $V \colon \Ab \longrightarrow \Set$ is the 
  forgetful functor; let $\beta \in F(Y)$ be the
  element corresponding to~$\id_Y \in \Mor_C(Y,Y)$.

  Let $X \in \Ob(C)$, let $\alpha \in F(X)$, and let $d \in \Z$. Then
  $d \cdot \alpha$ corresponds to a morphism~$f \in \Mor_C(Y,X)$, and so
  $ d \cdot \alpha = F(f) (\beta)$. Thus, $D(\beta,\alpha) = \Z$, which 
  shows that $\alpha$ is weakly flexible with respect to~$F$.
\end{proof}

\begin{exa}[homotopy groups]
  Let $n \in \N_{>1}$. Then the $n$-th homotopy group functor
  \[ \pi_n \colon \Toppth \longrightarrow \Ab 
  \]
  on the homotopy category~$\Toppth$ of pointed topological spaces is
  represented by the marked $n$-sphere. Hence, by
  Corollary~\ref{corffsnrep}, the functor~$\pi_n$ does \emph{not} admit a
  non-trivial finite homogeneous functorial semi-norm.  
\end{exa}

\begin{exa}[cohomology theories on CW-complexes]
  Cellular cohomology with coefficients in an Abelian group~$A$ does
  \emph{not} admit a non-trivial finite homogeneous functorial semi-norm in
  degree~$n \in \N$: The functor~$H^n(\args;A)$ on the homotopy
  category~$\CWh$ of CW-complexes is representable by CW-complexes of
  type~$K(A,n)$. 

  In view of the Brown representability theorem, the same holds for
  finitely additive generalised cohomology theories on the homotopy
  category of (pointed) finite CW-complexes that yield finitely
  generated Abelian groups on
  spheres~\cite[Chapter~21.8]{strom}\cite{adams}: In general, the
  representing object will not necessarily have the homotopy type of a
  finite CW-complex, but only of a CW-complex~$Y$ of finite type. If
  $X$ is a finite CW-complex of dimension~$n$, then the
  inclusion~$Y_{n+1} \hookrightarrow Y$ of the $(n+1)$-skeleton
  induces a natural bijection~$[X,Y_{n+1}] \longrightarrow [X,Y]$ and
  we can apply Corollary~\ref{corffsnrep} to the restriction of our
  given cohomology functor~$F$ to the full subcategory of~$\CWh$
  generated by $X$ and~$Y_{n+1}$; on this subcategory, $F$ is
  naturally isomorphic to~$[\args, Y_{n+1}]$ This shows that all
  classes in~$F(X)$ are weakly flexible. Hence, any finite homogeneous
  functorial semi-norm on~$F$ is trivial on~$F(X)$.

  For example, this applies to complex/real topological $K$-theory or
  oriented cobordism.

  The case of countably additive cohomology theories 
  can be handled as in Section~\ref{subsecwfgc}.
\end{exa}

\begin{exa}[bounded cohomology of spaces and groups] 
  Let $n \in \N_{>1}$. Then there exist finite CW-complexes~$X$ and
  (finitely presented) groups~$G$ respectively such that
  $\|\cdot\|_\infty$ is non-trivial on~$H^n_b(X;\R)$ and
  $H^n_b(G;\R)$, respectively; for example, one could take the
  cohomological fundamental class of oriented closed connected
  hyperbolic $n$-manifolds and their fundamental groups,
  respectively~\cite{vbc}. Hence, the bounded cohomology
  functors~$H^n_b(\args;\R)$ are \emph{not} representable on the
  \begin{itemize}
    \item (homotopy) category of topological spaces 
    \item (homotopy) category of finite CW-complexes
    \item category of groups
    \item category of finitely generated groups
    \item category of finitely presented groups
    \item \dots
  \end{itemize}
  Notice that in view of the mapping theorem for bounded
  cohomology~\cite{vbc,ivanov}, the pictures for the case of
  path-connected spaces and groups are equivalent.
\end{exa}


\subsection{Weak flexibility and countable additivity}\label{subsecwfgc}

\begin{cor}[finite functorial semi-norms and countable additivity]
  \label{corffsnlim}
  Let $C$ be a category such that for every object~$X \in \Ob(C)$ the
  countably infinite coproduct~$\coprod_{\N} X$ exists in~$C$. Let $F
  \colon C \longrightarrow \Ab$ be a contravariant functor such that
  for all $X \in \Ob(C)$ the structure maps~$X \longrightarrow
  \coprod_\N X$ of the summands induce an isomorphism
  \[ \varphi_X \colon F\Bigl(\coprod_\N X\Bigr) \longrightarrow \prod_\N F(X). 
  \]
  If $\widehat F \colon C \longrightarrow \snAbfin$ is
  a finite homogeneous functorial semi-norm on~$C$, then $\widehat F$ is trivial.
\end{cor}
\begin{proof}
  Let $\alpha \in F(X)$. It suffices to show that $\alpha$ is weakly flexible 
  with respect to~$F$ (Proposition~\ref{propfund}). Let $Y := \coprod_\N X$, 
  and let 
  \[ \beta := \varphi_X^{-1} \bigl( (d \cdot \alpha)_{d\in \N} \bigr) \in F(Y).
  \]
  For $d \in \N$ the structure map~$f_d \in \Mor_C(X, \coprod_\N X)$ of the 
  $d$-th summand satisfies
  \begin{align*} 
    F(f_d)(\beta)  = F(f_d) \bigl( \varphi_X^{-1}  \bigl( (d \cdot \alpha)_{d\in \N} \bigr)\bigr)
                   = F(\id_X) (d \cdot \alpha) 
                   = d \cdot \alpha.
  \end{align*}
  Hence, $\alpha$ is weakly flexible with respect to~$F$.
\end{proof}

\begin{exa}[additive cohomology theories]
  Corollary~\ref{corffsnlim} applies to all generalised cohomology
  theory functors~$\Top \longrightarrow \Ab$ that
  satisfy the countable additivity axiom; this includes, in particular, 
  for all~$n \in \N$ and all Abelian groups~$A$ singular cohomology 
  $H^n(\args;A) \colon \Top \longrightarrow \Ab$ in degree~$n$. 
\end{exa}

\begin{exa}[group cohomology]
  Group cohomology does \emph{not} admit a non-trivial finite homogeneous 
  functorial semi-norm in non-zero degree: In the category of groups,
  all coproducts exist (given by free products of groups), and for~$n \in \N_{>0}$
  and all Abelian groups~$A$ the group cohomology functor~$H^n(\args;A)
  \colon \Group \longrightarrow \Ab$ satisfies the corresponding
  compatibility condition of Corollary~\ref{corffsnlim}. Notice that
  $H^n(\args;A) \colon \Group \longrightarrow \Ab$ in general is \emph{not}
  representable if~$n > 1$.
\end{exa}

\begin{exa}[comparison maps]\label{exa:compmaps}
  Let $n \in \N_{>1}$.  Because Abelian groups are amenable, we have
  that $H^n_b(K(\Z^n,1);\R) \cong 0$ (which follows by averaging
  over invariant means~\cite{vbc,ivanov}), but $H^n(K(\Z^n,1);\R)
  \cong \R$. This simple example shows that the comparison
  map~$H^n_b(\args;\R) \Longrightarrow H^n(\args;\R)$ in general is
  not surjective on the category of topological spaces or groups,
  respectively.

  In contrast, we will now provide an argument that relies on the
  \emph{non-triviality} of the $\ell^\infty$-semi-norm: Let $C$ be a
  category of spaces or groups having the following properties:
  \begin{enumerate}
    \item\label{propy1}
      The functor~$H^n(\args;\R) \colon C \longrightarrow \Ab$ does
      not admit a non-trivial finite homogeneous functorial semi-norm,
      and
    \item\label{propy2} there exist $X \in \Ob(C)$ and~$\varphi \in
      H^n(X;\R)$ with~$0 < \|\varphi\|_\infty < \infty$.
  \end{enumerate}
  E.g., we could take the category of topological spaces or the
  category of groups. Moreover, let $c \colon H^n_b(\args;\R)
  \Longrightarrow H^n(\args;\R)$ be the comparison map on~$C$
  (Example~\ref{exabc}). Then, for all~$X \in \Ob(C)$ we obtain a
  homogeneous semi-norm
  \begin{align*}
    H^n(X;\R) & \longrightarrow \R_{\geq 0}\cup\{\infty\} \\
    \varphi & \longmapsto
    \inf \bigl\{ \|\psi\|_\infty 
         \bigm| \psi \in H^n_b(X;\R), c_X(\psi) = \varphi 
         \bigr\}
  \end{align*}
  on~$H^n(X;\R)$, which is finite if and only if $c_X \colon
  H^n_b(X;\R) \longrightarrow H^n(X;\R)$ is surjective (recall that
  $\inf \emptyset = \infty$). By construction of bounded cohomology,
  this semi-norm coincides with~$\|\cdot\|_\infty$
  on~$H^n(X;\R)$. Hence, property~(2) implies that this functorial
  semi-norm on~$H^n(\args;\R) \colon C \longrightarrow \Ab$ is
  non-trivial.  Therefore, property~(1) implies that there is an~$X
  \in \Ob(C)$, for which the comparison map~$c_X \colon
  H^n_b(X;\R) \longrightarrow H^n(X;\R)$ is \emph{not}
  surjective.

  It would be interesting to know whether such arguments could also be
  successfully applied to continuous bounded cohomology of restrictive
  classes of topological groups.
\end{exa}

\begin{exa}[$KK$-theory of separable $C^*$-algebras]
  Let $B$ be a separable $C^*$-algebra. Then $KK(\args,B) \colon \Cst
  \longrightarrow \Ab$ satisfies the countable additivity condition of
  Corollary~\ref{corffsnlim}~\cite[Theorem~19.7.1]{blackadar}, where
  $\Cst$ denotes the category of separable $C^*$-algebras. Hence,
  $KK(\args,B) \colon \Cst \longrightarrow \Ab$ does \emph{not} admit
  a non-trivial finite homogeneous functorial semi-norm.
\end{exa}

\section{Weak flexibility in singular homology and $\ell^1$-homology}\label{secsingl1}

We will give characterisations of weak flexibility in singular
homology (Section~\ref{subsecwfsing}) and $\ell^1$-homology
(Section~\ref{subsecwfl1}) in terms of $\ell^1$-semi-norms. Moreover,
we will show that flexibility in $\ell^1$-homology is equivalent to
vanishing of classes (Section~\ref{subsecflexl1}).

\subsection{Weak flexibility in singular homology}\label{subsecwfsing}

Weak flexibility in singular homology can be characterised as follows:

\begin{thm}[weak flexibility in singular homology]\label{thmwfsing}
  Let $X$ be a topological space, let $n \in \N$, and let $\alpha \in
  H_n(X;\Z)$. Then the following are equivalent:
  \begin{enumerate}
    \item The class~$\alpha$ is weakly flexible with respect
      to~$H_n(\args;\Z)$.
    \item The sequence
      $\bigl(\| d \cdot \alpha \|_{1,\Z} \bigr)_{d\in\N}
      $
      contains a bounded subsequence.
  \end{enumerate}
\end{thm}

For the proof of the theorem we will use that singular homology
classes can be represented essentially in terms of combinatorial types
that describe how faces of singular simplices in a singular chain can
cancel in the singular chain complex:

\begin{defi}[combinatorial types and associated chains]
  \hfil
  \begin{itemize}
  \item
    Let $k, n \in \N$. We then write
    \[ T(k,n) := P\bigl( (\{1,\dots,k\} \times \{0,\dots,n\})^{\times 2}
                  \bigr)
    \]
    for the set of all relations on~$\{1,\dots,k\} \times \{0,\dots,n\}$. 
  \item
    For~$t \in T(k,n)$, we define the topological space
    \[ Y_t := \{1,\dots,k\} \times \Delta^n \bigm/ \sim_t, 
    \]
    where ``$\sim_t$'' is the equivalence relation generated by
    \begin{align*} (j,x) \sim_t (j',x') 
       \Longleftrightarrow 
       \exi{y \in \Delta^{n-1}} \exi{\ell,\ell' \in \{0,\dots,n\}}
       & i_\ell(y) = x \\ 
       \land\; & i_{\ell'}(y) = x'\\ 
       \land\; & \bigl((j,\ell), (j',\ell')\bigr) \in t.
    \end{align*}
    Here, $i_\ell \colon \Delta^{n-1} \longrightarrow \Delta^n$ denotes 
    the inclusion of the $\ell$-th face.
  \item For $t \in T(k,n)$ and $\varepsilon \in \{-1,0,1\}^k$ we consider 
    the singular chain
    \[ z_{t,\varepsilon} := \sum_{j=1}^k \varepsilon_j \cdot \tau_j 
       \in C_n(Y_t;\Z),
    \]
    where $\tau_j \colon \Delta^n \longrightarrow Y_t$ is the singular
    simplex induced from the inclusion~$\Delta^n \hookrightarrow
    \{1,\dots,k\} \times \Delta^n$ as $j$-th component.
  \end{itemize}
\end{defi}

\begin{lem}[representing singular classes through combinatorial types]\label{lemcombtypes}
  Let $X$ be a topological space, and let $c = \sum_{j=1}^k
  \varepsilon_j \cdot \sigma_j \in C_n(X;\Z)$ be a cycle with
  $\varepsilon_1, \dots, \varepsilon_k \in \{-1,0,1\}$. Let 
  \begin{align*} 
    t:= \bigl\{ ((j,\ell), (j',\ell')) 
         \bigm| \;& j, j' \in \{1,\dots,k\},\  
                  \ell,\ell'\in \{0,\dots, n\},\\ 
                & \sigma_j \circ i_\ell = \sigma_{j'} \circ i_{\ell'}
         \bigr\} \in T(k,n),
  \end{align*}
  and let $f_c \colon Y_t \longrightarrow X$ be the continuous map 
  induced from the singular simplices~$\sigma_1, \dots, \sigma_k \colon \Delta^n \longrightarrow X$. 
  Then $f_c$ is well-defined, $z_{t,\varepsilon} \in C_n(Y_t;\Z)$ is a cycle, 
  and
  \[ H_n(f_c;\Z)[z_{t,\varepsilon}] = [c]  
  \]
  holds in~$H_n(X;\Z)$.
\end{lem}
\begin{proof}
  All these properties directly follow from the construction. 
\end{proof}

The proof of Theorem~\ref{thmwfsing} now is only a matter of counting:

\begin{proof}[Proof (of Theorem~\ref{thmwfsing})]
  Suppose that the class~$\alpha$ is weakly flexible with respect
  to~$H_n(\args;\Z)$, i.e., there is a space~$Y$ and a class~$\beta
  \in H_n(Y;\Z)$ such that $D(\beta,\alpha)$ is infinite. Let $d \in
  D(\beta,\alpha)$; then there is a map~$f \colon Y \longrightarrow X$
  with~$H_n(f;\Z)(\beta) = d \cdot \alpha$; in particular, 
  \[ \bigl\| |d| \cdot \alpha \bigr\|_{1,\Z} 
     = \bigl\| H_n(f;\Z)(\beta) \bigr\|_{1,\Z} 
     \leq \|\beta\|_{1,\Z}.
  \]
  Because $D(\beta,\alpha)$ is infinite, $(\|d \cdot \alpha
  \|_{1,\Z})_{d \in \N}$ thus contains a subsequence that is bounded
  by~$\|\beta\|_{1,\Z}$.

  Conversely, suppose that the sequence $\bigl(\| d \cdot \alpha
  \|_{1,\Z} \bigr)_{d\in\N} $ contains a subsequence,
  say~$\bigl(\| d_m \cdot \alpha \|_{1,\Z} \bigr)_{m\in\N}$, bounded
  by~$k \in \N$. For $m \in \N$ let $c_m \in C_n(X;\Z)$ be a cycle
  representing~$d_m \cdot \alpha$ with $\|c_m\|_{1,\Z} \leq k$;
  we may assume that $c_m$ is of the form
  \[ c_m = \sum_{j=1}^k \varepsilon_{m,j} \cdot \sigma_{m,j}  
  \]
  with~$\varepsilon_{m,1}, \dots, \varepsilon_{m,k} \in \{-1,0,1\}$. 
  Moreover, let $t_m \in T(k,n)$ be the corresponding combinatorial 
  type of~$c_m$ as in Lemma~\ref{lemcombtypes}. Because the set 
  \[ T(k,n) \times \{-1,0,1\}^k 
  \]
  is finite, there is a subsequence of~$(c_m)_{m \in \N}$ of chains that have
  the same combinatorial type~$t \in T(k,n)$ and the same
  coefficients~$\varepsilon \in \{-1,0,1\}^k$. For any index~$m \in \N$ 
  of this subsequence, we have 
  \[ H_n(f_{c_m};\Z)[z_{t,\varepsilon}] = H_n(f_{c_m};\Z)[z_{t_m,\varepsilon_{m}}] 
     = [c_m] = d_m \cdot \alpha,
  \]
  where $f_{c_m} \colon Y_{t_m} = Y_t \longrightarrow X$ is the map
  from Lemma~\ref{lemcombtypes}. Thus,~$D([z_{t,\varepsilon}],\alpha)$ 
  is infinite, i.e., $\alpha$ is weakly flexible with respect 
  to~$H_n(\args;\Z)$.
\end{proof}

\begin{rem} 
  The second condition of Theorem~\ref{thmwfsing} is related to the vanishing of
  the $\ell^1$-semi-norm on singular homology with $\R$-coefficients
  as follows: If $X$ is a topological space, then for all~$n\in \N$
  the change of coefficients map~$H_n(X;\Q) \longrightarrow H_n(X;\R)$
  is isometric with respect to corresponding
  $\ell^1$-semi-norms~\cite[Lemma~2.9]{mschmidt}. Rearranging denominators hence
  shows that
  \[ \| \alpha_\R \|_1 = \inf_{d \in \N_{>0}} \frac1d \cdot \| d \cdot \alpha\|_{1,\Z}\]
  for all~$\alpha \in H_n(X;\Z)$, where $\alpha_\R \in H_n(X;\R)$ denotes 
  the image of~$\alpha$ under the change of coefficients map. 
  
  This is the first step in a programme that defines and studies 
  secondary invariants associated with the $\ell^1$-semi-norm and 
  simplicial volume.
\end{rem}

\begin{cor}[weak flexibility in singular homology, manifold case]\label{cor:wfmfd}
  Let $M$ be an oriented closed connected $n$-manifold. Then the 
  following are equivalent:
  \begin{enumerate}
    \item There exists an oriented closed connected $n$-manifold~$N$ 
      such that the set $\bigl\{ \deg f \bigm| f \in \map(N,M)\bigr\}$ of 
      mapping degrees is infinite.
    \item The fundamental class~$[M]$ is weakly flexible with respect 
      to~$H_n(\args;\Z)$. 
    \item The $\R$-fundamental class~$[M]_\R$ is 
      weakly flexible with respect to~$H_n(\args;\R)$.
    \item The sequence $\bigl( \| d \cdot [M] \|_{1,\Z}\bigr)_{d \in \N}$ contains 
      a bounded subsequence.
    \item All finite homogeneous functorial semi-norms 
      on~$H_n(\args;\R)$ are zero on~
      $[M]_\R$.
  \end{enumerate} 
\end{cor}
\begin{proof}
  By definition of the mapping degree and the universal coefficient
  theorem, the first three conditions are equivalent. In view of
  Theorem~\ref{thmwfsing}, condition~(2) and condition~(4) are
  equivalent. Proposition~\ref{propfund} and explicit
  constructions~\cite[Section~7.1]{crowleyloeh} of functorial
  semi-norms on~$H_n(\args;\R)$ show that condition~(1) and condition~(5) are
  equivalent.
\end{proof}

More specifically, the case of bound~$1$ in
Corollary~\ref{cor:wfmfd}(4) is equivalent to domination by an
odd-dimensional sphere~\cite[Theorem~3.2]{smallisv}.

\begin{rem}[weak flexibility in group homology]
  Gaifullin's construction of aspherical URC
  manifolds~\cite{gaifullinurc} and the fact that $\|\cdot\|_{1,\Z}$
  on integral singular homology of classifying/aspherical spaces
  coincides with~$\|\cdot\|_{1,\Z}$ on integral group homology
  of the corresponding fundamental group (Example~\ref{exal1groups}) 
  allow to translate Theorem~\ref{thmwfsing} also into the
  corresponding characterisation for weak flexbility with respect
  to~$H_n(\args;\Z) \colon \Group \longrightarrow \Ab$.
\end{rem}

\subsection{Weak flexibility in $\ell^1$-homology}\label{subsecwfl1}

We will now discuss weak flexibility in $\ell^1$-homology of
simplicial sets.  In particular, the following discussions will apply
to $\ell^1$-homology of spaces and groups.

\begin{setup}\label{setupl1}
  Let $C$ be a category, let $S \colon C \longrightarrow \Delta(\Set)$
  be a functor, and $n\in \N_{>0}$. We abbreviate
  \[ H_n^{\ell^1,S} := H_n \circ \Ch_\R^{\ell^1} \circ S \colon C \longrightarrow \Ab.
  \]
  This functor is equipped with the finite functorial
  $\ell^1$-semi-norm (Example~\ref{exal1simpset}). For simplicity, we
  assume that $C$ contains an object~$\bullet$ and that the
  simplicial set~$S(\bullet)$ corresponds to the simplicial set~$P$
  satisfying~$P(n) = \{\emptyset\}$ for all~$n \in \N$ (i.e., the
  simplicial set corresponding to a one-point space).
\end{setup}

\begin{defi}[replicable, pointable]
  In the situation of Setup~\ref{setupl1}, an object~$X \in \Ob(C)$ 
  is called 
  \begin{itemize}
    \item \emph{replicable} if the coproduct~$\coprod_\N X$ exists in~$C$,
    \item \emph{pointable (with respect to~$S$)} if $\Mor_C(\bullet,X)$ 
      and $\Mor_C(X,\bullet)$ are non-empty.
  \end{itemize}
\end{defi}

All (non-empty) topological spaces are replicable and pointable (with
respect to the total singular complex functor) in the category of
topological spaces. All groups are replicable and pointable (with
respect to the simplicial classifying space functor) in the category
of groups.

In $\ell^1$-homology, weak flexibility is the universal reason for
vanishing of the $\ell^1$-semi-norm: 

\begin{thm}[weak flexibility in $\ell^1$-homology]\label{thmweakflexl1}
  In the situation of Setup~\ref{setupl1}, let $X \in \Ob(C)$ be an
  object that is replicable and pointable, and let $\alpha \in
  H_n^{\ell^1,S}(X)$.  Then the following are equivalent:
  \begin{enumerate}
    \item The class $\alpha$ is weakly flexible with respect
      to~$H_n^{\ell^1,S}$.
    \item All finite homogeneous functorial semi-norms on~$H_n^{\ell^1,S}$ are 
      zero on~$\alpha$.
    \item We have~$\|\alpha\|_1 = 0$.
  \end{enumerate}
\end{thm}

\begin{proof}
  The first condition implies the second one by Proposition~\ref{propfund}. 
  Clearly, the second condition implies the third. 

  For the remaining implication, suppose that $\|\alpha\|_1 = 0$. We show that 
  $\alpha$ is weakly flexible with respect to~$H_n^{\ell^1,S}$:
  Because $X$ is pointable there exist morphisms~$i \in \Mor_C(\bullet,X)$
  and $q \in \Mor_C(X,\bullet)$. Let
  \[ p := i \circ q \in \Mor_C(X,X). 
  \]
  This morphism allows us to neglect certain $\ell^1$-chains: 
  Using that $n \in \N_{>0}$, that $p$ factors over~$\bullet$ and 
  spelling out~$\Ch_\R^{\ell^1}(S(\bullet)) = \Ch_\R^{\ell^1}(P)$ explicitly
  yields: 
  \begin{itemize}
    \item[($*$)]
    For any cycle~$c \in \Ch_\R^{\ell^1}(S(X))$ 
    there is a chain~$b \in \Ch_\R^{\ell^1}(S(X))$ satisfying
    \[ \partial_{n+1} b = \Ch_\R^{\ell^1}(S(p))(c) 
    \quad\text{and}\quad
    \| b \|_1 \leq \|c\|_1.
    \]
  \end{itemize}

  We now construct a suitable $\ell^1$-homology class on~$Y :=
  \coprod_\N X$: For~$k \in \N$ let $i_k \in \Mor_C\bigl(X,  
  \coprod_\N X\bigr)$ be the structure map of the $k$-th summand. Using the
  universal property of the coproduct, for every~$m\in \N$ we obtain a
  morphism~$p_m \in \Mor_C(Y,X)$ satisfying for all~$k \in \N$:
  \[ 
     p_m \circ i_k = 
     \begin{cases} 
       \id_X & \text{if $k = m$}\\
       p     & \text{if $k \neq m$}.
     \end{cases}
  \]
  Because of $\|\alpha\|_1 = 0$, for every~$k \in \N$ 
  there is a cycle~$c_k \in \Ch_\R^{\ell^1} \circ S(X)$ representing~$\alpha$ 
  with
  \[ \|c_k\|_1 \leq \frac 1{2^k}. 
  \]
  Hence, 
  \[ c := \sum_{k \in \N} k \cdot \Ch_\R^{\ell^1}\bigl(S (i_k)\bigr) (c_k)
  \]
  is a well-defined $\ell^1$-chain in~$\Ch_\R^{\ell^1}(S(Y))$, which is a cycle. 
  We consider the corresponding class~$\beta := [c] \in H_n^{\ell^1,S}$.

  The morphisms~$(p_k)_{k \in \N}$ witness that~$D(\beta, \alpha)$ is infinite:
  Because of property~($*$), for every~$k \in \N$ there is a chain~$b_k \in \Ch_\R^{\ell^1}(S(X))$ with
  \[ \partial_{n+1} b_k = \Ch_\R^{\ell^1}(S(p))(c_k) 
    \quad\text{and}\quad
    \| b_k \|_1 \leq \|c_k\|_1 \leq \frac 1{2^k}.
  \]
  Therefore, for all~$m \in \N$ we have in~$H_n^{\ell^1,S}(X)$:
  \begin{align*}
    H_n^{\ell^1,S}(p_m) (\beta)
    & = \biggl[ \sum_{k \in \N} k \cdot \Ch_\R^{\ell^1}\bigl(S(p_m \circ i_k)\bigr)(c_k) 
        \biggr] 
    \\
    & = [ m \cdot c_m] 
      + \biggl[\sum_{k \in \N \setminus \{m\}} k \cdot \partial_{n+1} b_k  
        \biggr]
    \\ 
    & = m \cdot [c_m] + 0
    \\
    & = m \cdot \alpha.
  \end{align*}
  In the third step, we used that $\sum_{k \in \N \setminus \{m\}} k
  \cdot \partial_{n+1} b_k$ is indeed an $\ell^1$-chain and that 
  $\sum_{k \in \N \setminus \{m\}} k
  \cdot \partial_{n+1} b_k = \partial_{n+1}\bigl(\sum_{k \in \N \setminus\{m\}} k \cdot b_k \bigr)$. 

  Hence, $\alpha$ is weakly flexible with respect to~$H_n^{\ell^1,S}$.
\end{proof}

Moreover, using the same type of arguments, we also obtain that the $\ell^1$-semi-norm 
is the ``maximal'' functorial semi-norm on $\ell^1$-homology:

\begin{prop}[finite functorial semi-norms on $\ell^1$-homology]
  In the situation of Setup~\ref{setupl1}, suppose that all countable
  coproducts exist in~$C$ and that all objects in~$C$ are pointable.
  Then any finite homogeneous functorial semi-norm on~$H_n^{\ell^1,S}$
  is dominated by a multiple of the $\ell^1$-semi-norm. 
\end{prop}

\begin{proof}
  Let $|\cdot|$ be a finite homogeneous functorial semi-norm on~$H_n^{\ell^1,S}$. 
  \emph{Assume} for a contradiction that $|\cdot|$ is \emph{not} 
  dominated by a multiple of the $\ell^1$-semi-norm. That is, for every~$k \in \N$ 
  there exists an object~$X_k \in \Ob(C)$ and a class~$\alpha_k \in H_n^{\ell^1,S}(X_k)$ satisfying
  \[ |\alpha_k| > k \cdot 2^k \cdot \|\alpha_k\|_1. 
  \]
  By Theorem~\ref{thmweakflexl1}, $\|\alpha_k\|_1 \neq 0$; moreover,
  because $|\cdot|$ and $\|\cdot\|_1$ are homogeneous, we can renormalise 
  these classes in such a way that we may assume in addition that
  \[ \frac 12 \leq \|\alpha_k\|_1 \leq 1 
  \]
  for all~$k \in \N$. In particular, for every~$k \in \N$ there is 
  a cycle~$c_k \in \Ch_\R^{\ell^1}(S(X_k))$ representing~$\alpha_k$ 
  with~$\|c_k\|_1 \leq 2$. 

  Let $X := \coprod_{k \in \N} X_k$ and let $i_k \in \Mor_C(X_k,X)$ be the 
  structure map of the $k$-th summand. Furthermore, let $q_k \in \Mor_C(X_k,\bullet)$ 
  and $j_k \in \Mor_C(\bullet,X_k)$. Using the universal property of the coproduct, 
  for every~$m\in \N$ we obtain a morphism~$p_m \in \Mor_C(X,X_m)$ satisfying for 
  all~$k \in \N$:
  \[ p_m \circ i_k = 
    \begin{cases}
      \id_{X_m} & \text{if $k = m$}\\
      j_m \circ q_k & \text{if $k \neq m$}.
    \end{cases}
  \]
  Then 
  \[ \beta := \biggl[ \sum_{k \in \N} \frac 1{2^k} \cdot \Ch_\R^{\ell^1}\bigl(S(i_k)\bigr)(c_k)
              \biggr] 
     \in H_n^{\ell^1,S}(X)
  \]
  is a well-defined $\ell^1$-homology class and the same argument as in the proof 
  of Theorem~\ref{thmweakflexl1} shows that 
  \[ H_n^{\ell^1,S}(p_m)(\beta) = \frac 1{2^m} \cdot \alpha_m
  \] 
  for all~$m \in \N$. Therefore, for all~$m\in \N$ we obtain
  \begin{align*}
    |\beta| & \geq \bigl| H_n^{\ell^1,S}(p_m)(\beta) \bigr| 
              = \frac1{2^m} \cdot |\alpha_m| 
              > m \cdot \|\alpha_m\|_1 
              \geq \frac 12 \cdot m
  \end{align*}
  which contradicts~$|\cdot|$ being finite.
\end{proof}

The corresponding statement for singular homology is \emph{not} true:
it is possible to ``distort'' the $\ell^1$-semi-norm in such a way
that the resulting functorial semi-norm is finite, but \emph{not}
dominated by a multiple of the $\ell^1$-semi-norm~\cite[proof of
  Theorem~5.7]{crowleyloeh}.

Moreover, for the sake of completeness, we also mention the following,
qualitative, version of Lemma~\ref{lemcombtypes} for $\ell^1$-homology
of spaces:

\begin{lem}[countable geometric support of $\ell^1$-homology classes]
  Let $X$ be a topological space, let $n \in \N$, and let $\alpha \in
  \loneh_n(X;\R)$. Then there exists a countable simplicial
  complex~$K$, a class~$\beta \in \loneh_n(|K|;\R)$ and a continuous
  map~$f \colon |K| \longrightarrow X$ with
  \[ \loneh_n (f;\R)(\beta) = \alpha 
     \quad\text{and}
     \quad
     \|\beta\|_1 = \|\alpha\|_1.
  \]
\end{lem}
\begin{proof}
  We give a direct, geometric proof: Let $(c_k)_{k \in \N} \in C^{\ell^1}_n(X;\R)$
  be a sequence of cycles representing~$\alpha$ in~$\loneh_n(X;\R)$ with
  \[ \|\alpha\|_1 = \inf_{k \in \N} \|c_k\|_1,
  \]
  and let $(b_k)_{k \in \N} \in C_{n+1}^{\ell^1}(X;\R)$ satisfying
  \begin{align}\label{bdyeq} 
    \partial_{n+1} b_k = c_0 - c_k; 
  \end{align}
  for these chains~$(c_k)_{k \in \N}$ and $(b_k)_{k \in \N}$, in total
  only a countable set of singular simplices of dimension~$k$ or~$k+1$
  is needed.  Similarly to Lemma~\ref{lemcombtypes} and the discussion
  preceding that lemma, we can construct a topological space~$Y$ by
  gluing a countable set of copies of~$\Delta^n$ and~$\Delta^{n+1}$
  according to the combinatorics of the relations between the
  respective faces given by the fact that the~$(c_k)_{k \in \N}$ are
  cycles and by Equation~\eqref{bdyeq}; moreover, using the simplices
  in~$Y$ and the same coefficients as in the original chains on~$X$,
  we obtain corresponding chains~$(\overline c_k)_{k \in \N} \subset
  C_k^{\ell^1}(Y;\R)$, and $(\overline b_k)_{k \in \N} \subset
  C_{k+1}^{\ell^1}(Y;\R)$, as well as a continuous map~$f \colon Y
  \longrightarrow X$ (constructed out of the singular simplices of the
  original chains) satisfying
  \begin{align*}
    \|\overline c_k\|_1 & = \|c_k\|_1, \\
    \partial_{n} \overline c_k & = 0,\\
    \partial_{n+1} \overline b_k & = \overline c_0 - \overline c_k,\\
    C^{\ell^1}_n(f) (\overline c_0) &= c_0.
  \end{align*}
  Hence, $\beta \in H_n(Y;\R)$ satisfies~$\loneh_n(f;\R)(\beta) = \alpha$. 
  Functoriality of the $\ell^1$-se\-mi-norm and the above relations between
  the constructed cycles on~$Y$ show that $\|\beta\|_1 = \|\alpha\|_1$.

  The double barycentric subdivision of the simplices
  involved in the construction of~$Y$ yields a countable 
  simplicial complex~$K$ with~$|K| \cong Y$.
\end{proof}

Already in degree~$0$ one can see that not every $\ell^1$-homology 
class of any topological space can be isometrically represented by 
an $\ell^1$-homology class of a finite simplicial complex.

\subsection{Flexibility in $\ell^1$-homology}\label{subsecflexl1}

On the other hand, in $\ell^1$-homology, flexibility of a class is the 
same as being trivial; clearly, this behaviour is very
different from ordinary homology.

\begin{thm}[flexibility in $\ell^1$-homology]
  In the situation of Setup~\ref{setupl1}, let $X \in \Ob(C)$, and let
  $\alpha \in H_n^{\ell^1,S}(X)$.  Then the following are equivalent:
  \begin{enumerate}
    \item The class~$\alpha$ is flexible with respect to~$H_n^{\ell^1,S}$.
    \item The class~$\alpha$ is dominated by a class that is flexible
      with respect to~$H_n^{\ell^1,S}$, i.e., there is 
      an object~$Y \in \Ob(C)$, a flexible class~$\beta \in H_n^{\ell^1,S}(Y)$,  
      and a morphism~$f \in \Mor_C(Y,X)$ with $H_n^{\ell^1,S}(f) (\beta) = \alpha$.
    \item We have $\alpha = 0 \in H_n^{\ell^1,S}(X)$.
  \end{enumerate}
\end{thm}
\begin{proof}
  Clearly, the third condition implies the first, and the first
  condition implies the second.

  Suppose that the second condition holds, i.e., there is $Y
  \in \Ob(C)$, a flexible class~$\beta \in H_n^{\ell^1,S}(Y)$, and 
  $f \in \Mor_C(Y,X)$ with $H_n^{\ell^1,S}(f) (\beta) =
  \alpha$. By functoriality, it suffices to show that $\beta = 0 \in
  H_n^{\ell^1, S}(Y)$. 

  Because $\beta$ is flexible with respect to~$H_n^{\ell^1,S}$ there is 
  a $d\in \Z$ with $|d| \geq 2$ and an endomorphism~$g \in \Mor_C(Y,Y)$ satisfying
  \[ H_n^{\ell^1,S}(g)(\beta) = d \cdot \beta. 
  \]
  Let $c \in \Ch_\R^{\ell^1}(S(Y))$ be a cycle
  representing the class~$\beta$. Hence, there is a chain~$b \in
  \Ch_\R^{\ell^1}(S(Y))$ with
  \[ \partial_{n+1} (b) = c - \frac1d \cdot \Ch_\R^{\ell^1}\bigl(S(g)\bigr)(c). 
  \]
  Then the ``geometric series''
  \[ b' := \sum_{k \in \N} \frac1{d^k} \cdot \Ch_\R^{\ell^1}\bigl(S(g^{\circ k})\bigr)(b) 
  \]
  is an $\ell^1$-chain. Rearranging absolutely convergent series together with 
  functoriality of~$\Ch_\R^{\ell^1} \circ S$ shows that 
  \begin{align*} 
    \partial_{n+1} b' & = \sum_{k \in \N} \frac 1{d^k} \cdot \Ch_\R^{\ell^1}\bigl(S(g^{\circ k})\bigr)(c)
                  - \sum_{k \in \N} \frac 1{d^{k+1}} \cdot \Ch_\R^{\ell^1}\bigl(S(g^{\circ {k+1}})\bigr)(c)
                 = c
  \end{align*}
  holds in~$\Ch_\R^{\ell^1}(S(Y))$. 
  Therefore,~$\beta = [c] = [\partial_{n+1} b'] = 0$ in~$H^{\ell^1,S}_n(Y)$.
\end{proof}

So, in $\ell^1$-homology, the difference between trivial
$\ell^1$-semi-norm and vanishing of the class corresponds to the
difference between weak flexibility and flexibility. It is an open
problem to determine whether for any topological space~$X$, any~$n \in
\N$ and any~$\alpha \in H_n(X;\R)$ with~$\|\alpha\|_1 = 0$ the image
of~$\alpha$ in~$H_n^{\ell^1}(X;\R)$ under the comparison map between
singular and $\ell^1$-homology can be non-trivial. This problem plays a
role in the context of simplicial volume of non-compact
manifolds~\cite[Theorem~6.4]{loehl1}.


\medskip
\vfill

\noindent
\emph{Clara L\"oh}\\[.5em]
  {\small
  \begin{tabular}{@{\qquad}l}
    Fakult\"at f\"ur Mathematik\\
    Universit\"at Regensburg\\
    93040 Regensburg\\
    Germany\\
    \textsf{clara.loeh@mathematik.uni-regensburg.de}\\
    \textsf{http://www.mathematik.uni-regensburg.de/loeh}
  \end{tabular}}
\end{document}